\documentclass[reqno, twoside, a4paper]{amsart}

\usepackage{amsmath}
\usepackage{amssymb}
\usepackage{amsthm}
\usepackage{graphicx}
\usepackage{subfig}
\usepackage{enumerate}
\usepackage{hyperref}

\hypersetup{%
pdftitle={Symplectic rational blow-ups on rational 4-manifolds},%
pdfauthor={Heesang Park, Dongsoo Shin},%
pdfkeywords={rational blow-up, rational 4-manifold},%
colorlinks=true,%
citecolor=black,%
filecolor=black,%
linkcolor=black,%
urlcolor=black,%
}

\usepackage{comment}
\usepackage{xspace}
\usepackage{mathtools}

\usepackage{mathptmx}       
\usepackage{helvet}         
\usepackage{courier}        

\usepackage{tikz}

\usetikzlibrary{patterns,decorations.pathreplacing}
\usetikzlibrary{decorations.pathmorphing}

\tikzset{bullet/.style={
shape = circle,fill = black, inner sep = 0pt, outer sep = 0pt, minimum size = 0.35em, line width = 0pt, draw=black!100}}

\tikzset{circle/.style={
shape = circle,fill = none, inner sep = 0pt, outer sep = 0pt, minimum size = 0.35em, line width = 1pt, draw=black!100}}

\tikzset{rectangle/.style={
shape = rectangle,fill = white, inner sep = 0pt, outer sep = 0pt, minimum size = 0.35em, line width = 0pt, draw=black!100}}

\tikzset{empty/.style={
shape = circle,fill = white, inner sep = 0pt, outer sep = 0pt, minimum size = 0.35em, line width = 0pt, draw=white!100}}

\tikzset{xmark/.style={
shape = x,fill = white, inner sep = 0pt, outer sep = 0pt, minimum size = 0em, line width = 0pt, draw=white!100}}

\tikzset{longrectangle/.style={
inner sep = 1em,
rectangle,
minimum size=1em,
very thick,
draw=black!100, 
}}

\tikzset{label distance=-0.15em}

\tikzset{font=\scriptsize}

\usepackage{tikz-cd}
\tikzcdset{every label/.append style = {font = \normalsize}}

\newtheorem{theorem}{Theorem}[section]
\newtheorem{lemma}[theorem]{Lemma}
\newtheorem{proposition}[theorem]{Proposition}
\newtheorem{corollary}[theorem]{Corollary}

\newtheorem*{corollarywon}{Corollary}

\newtheorem*{theoremwon}{Theorem}

\theoremstyle{definition}

\newtheorem{remark}[theorem]{Remark}

\numberwithin{equation}{section}

\DeclarePairedDelimiter{\bilinear}{\langle}{\rangle}

\begin{document}

\title{Symplectic rational blow-ups on rational 4-manifolds}

\author[H. Park]{Heesang Park}

\address{Department of Mathematics, Konkuk University, Seoul 05029, Republic of Korea}

\email{HeesangPark@konkuk.ac.kr}

\author[D. Shin]{Dongsoo Shin}

\address{Department of Mathematics, Chungnam National University, Daejeon 34134, Republic of Korea}

\email{dsshin@cnu.ac.kr}



\begin{abstract}
We prove that if a symplectic 4-manifold $X$ becomes a rational 4-manifold after applying rational blow-down surgery, then the symplectic 4-manifold $X$ is originally rational. That is, a symplectic rational blow-up of a rational symplectic $4$-manifold is again rational. As an application we show that a degeneration of a family of smooth rational complex surfaces is a rational surface if the degeneration has at most quotient surface singularities, which generalizes slightly a classical result of [Bădescu, 1986] in algebraic geometry under a mild additional condition.
\end{abstract}

\maketitle

\section{Introduction}

A \emph{rational blow-down} is a surgery that replaces a regular neighborhood of a certain linear chains of negative 2-spheres (called $C_{p,q}$) with the corresponding rational homology ball (called $B_{p,q}$); for a brief summary on this surgery, see Section~\ref{section:rational-blow-down}. A rational blow-down surgery is a symplectic one. That is, if a symplectic 4-manifold $X$ contains a linear chain $C_{p,q}$ consisting of symplectic 2-spheres, then the resulting 4-manifold $X_{p,q}$ after applying a rational blow-down surgery to $X$ along $C_{p,q}$ is again a symplectic 4-manifold. So many interesting symplectic 4-manifolds have been constructed via a sequence of rational blow-down surgeries.

Especially exotic symplectic 4-manifolds (for instance, exotic $\mathbb{CP}^2 \sharp n \overline{\mathbb{CP}}^2$'s) have been constructed using rational blow-down surgeries, which is possible because a rational blow-down surgery \emph{is used to increase} the Kodaira dimension $\kappa(X)$ of an original symplectic 4-manifold $X$. Hence one may raise a question: Is it true that $\kappa(X) \le \kappa(X_{p,q})$ in any case? Conversely, is it possible that a rational blow-down surgery \emph{decrease} the Kodaira dimension in some case? But there are no known results so far, as far as the authors know.

In this paper we show a partial answer to the above question.

\begin{theoremwon}[Theorem~\ref{theorem:Main}]
Let $(X,\omega_X)$ be a symplectic 4-manifold. Suppose that a configuration $C_{p,q}$ is symplectically embedded in $X$; that is, all 2-spheres $u_i$ in $C_{p,q}$ are symplectically embedded in $X$ and they intersect positively. Let $(X_{p,q}, \omega_{p,q})$ be the symplectic rational blow-down of $(X,\omega_X)$ along $C_{p,q}$. If $X_{p,q}$ is rational, then so is $X$.
\end{theoremwon}

So a symplectic rational blow-down surgery cannot decrease the Kodaira dimension of a symplectic 4-manifold of with positive Kodaira dimension to $-\infty$. In terms of \emph{rational blow-up} (cf. Section~\ref{section:rational-blow-down}), it may be said that a symplectic rational blow-up of a rational symplectic $4$-manifold is again rational. It is easy to extend the above theorem to the case that $X$ contains many disjoint $C_{p,q}$'s; Corollary~\ref{corollary:multiple-cpq}.

On the other hand, one may characterize \emph{degenerations} of rational complex surfaces in algebraic geometry using the above result. Bădescu~\cite{Badescu-1986} classifies degenerations of rational complex surfaces into two classes; see Proposition~\ref{proposition:Badescu}. We generalize the result of Bădescu under a certain mild condition.

\begin{corollarywon}[Corollary~\ref{corollary:degeneration}]
Let $X$ be a degeneration of a family $\{X_t \mid t \in \Delta-\{0\}\}$ of rational surfaces. Suppose that the degeneration $X$ has at most quotient surface singularities. Then its minimal resolution $\widetilde{X}$ is also a rational surface.
\end{corollarywon}

In Section~\ref{section:rational-blow-down} we briefly review basics on rational blow-down surgery. We construct a special family of symplectic structures on certain symplectic 4-manifolds, which will be used in the proof of the main result Theorem~\ref{theorem:Main} given in Section~\ref{section:main-theorem}. We discuss some open problem also in Section~\ref{section:main-theorem}. Finally, we classify a degeneration of rational surfaces in Section~\ref{section:degenerations}.

\subsection*{Acknowledgements}

Heesang Park was supported by Basic Science Research Program through the National Research Foundation of Korea (NRF) funded by the Ministry of Education: NRF-2021R1F1A1063959. Dongsoo Shin was supported by research fund of Chungnam National University in 2020. The authors would like to thank Korea Institute for Advanced Study for warm hospitality when they were associate members in KIAS.

\section{The rational blow-down}\label{section:rational-blow-down}

A \emph{rational blow-down surgery} was developed by Fintushel--Stern~\cite{Fintushel-Stern-1997} and generalized by J. Park~\cite{JPark-1997}. Let $C_{p,q}$ ($1 \le q < p$) be a regular neighborhood of the linear chain of smooth $2$-spheres $u_i$ whose  dual graph is
\begin{equation*}
\begin{tikzpicture}
\node[bullet] (10) at (1,0) [label=above:{$-b_1$}] {};
\node[bullet] (20) at (2,0) [label=above:{$-b_2$}] {};

\node[empty] (250) at (2.5,0) [] {};
\node[empty] (30) at (3,0) [] {};

\node[bullet] (350) at (3.5,0) [label=above:{$-b_{r-1}$}] {};
\node[bullet] (450) at (4.5,0) [label=above:{$-b_r$}] {};

\draw [-] (10)--(20);
\draw [-] (20)--(250);
\draw [dotted] (20)--(350);
\draw [-] (30)--(350);
\draw [-] (350)--(450);
\end{tikzpicture}
\end{equation*}
with
\begin{equation*}
\frac{p^2}{pq-1} = b_1 - \cfrac{1}{b_2 - \cfrac{1}{b_3 - \cfrac{1}{\ddots - \cfrac{1}{b_r}}}}
\end{equation*}
for $b_i \ge 2$ ($ 1 \leq i \leq r$) in a smooth $4$-manifold $X$. One the other hand, a rational homology ball $B_{p,q}$ is a smooth $4$-manifold with the lens space $L(p^2, pq-1)$ as its boundary such that $H_{\ast}(B_{p,q};\mathbb{Q}) \cong H_{\ast}(B^4;\mathbb{Q})$.

One may cut $C_{p,q}$ from $X$ and paste $B_{p,q}$ along the boundary $L(p^2,pq-1)$ so that one obtains a new smooth $4$-manifold $X_{p,q}=(X-C_{p,q}) \cup_{L(p^2,pq-1)} B_{p,q}$. This surgery is called a \emph{rational blow-down surgery} along $C_{p,q}$.

The surgery can be performed compatibly with a symplectic structure, provided that the $2$-spheres $u_i$ are symplectic submanifolds of $X$; Symington~\cite{Symington-2001}. That is, if $X$ is a symplectic $4$-manifold and if each $2$-spheres $u_i$'s in $C_{p,q}$ are symplectic $2$-spheres intersecting positively with each other, then one can glue $B_{p,q}$ to $X-C_{p,q}$ along $\partial C_{p,q}$ so that the rational blow-down $X_{p,q}$ is also a symplectic $4$-manifold. Indeed such a cut-and-paste surgery can be performed in symplectic category for much general situations; See Park-Stipsicz~\cite{Park-Stipsicz-2014}.

Conversely, suppose that a symplectic $4$-manifold $Y$ contains a rational homology ball $B_{p,q}$. Then one can construct a new symplectic $4$-manifold $Y' = (Y-B_{p,q}) \cup_{L(p^2,pq-1)} C_{p,q}$ by replacing $B_{p,q}$ with $C_{p,q}$, which is called a \emph{symplectic rational blow-up surgery}.

\section{A special smooth family of symplectic structures}

We prove that there is a certain special smooth family of symplectic structures on a rationally blown-down 4-manifold that is rational.

\begin{proposition}\label{proposition:Key-proposition}
Assume that $X_{p,q}$ is rational and that $K_{\omega_{p,q}} \cdot \omega_{p,q} \ge 0$. Then there is a smooth family of symplectic structures $\omega_t$ ($t \in [0,1]$) on $X_{p,q}$ such that $\omega_0=\omega_{p,q}$ and $K_{\omega_1} \cdot \omega_1 < 0$.
\end{proposition}

We first construct certain families of symplectic structures on $X_{p,q}$ in Lemma~\ref{lemma:A}. Since $X_{p,q}$ is assumed to be rational, there is a symplectic structure $\widetilde{\omega}$ on $X_{p,q}$ such that $\widetilde{K} \cdot \widetilde{\omega} < 0$ by Proposition~\ref{proposition:McDuff-Salamon}. For each $t \in [0,1]$, we define a closed $2$-form $\lambda(t)$ by
\begin{equation*}
\lambda(t) = (1-t) \omega_{p,q} + t \widetilde{\omega}.
\end{equation*}
The set of the closed $2$-forms $\lambda(t)$ forms a line segment in the space of closed $2$-forms $\Omega^2(X_{p,q})$ on $X_{p,q}$, which is denoted by
\begin{equation*}
L_{\lambda} = \{\lambda(t) \mid 0 \le t \le 1 \}.
\end{equation*}
We denote by $\theta(t)=[\lambda(t)]$ the cohomology class of $\lambda(t)$ in $H^2(X_{p,q};\mathbb{R})$.

\begin{lemma}\label{lemma:A}
There is an increasing sequence $0 = t_0 < t_1 < \dotsb < t_k < t_{k+1}=1$ and, for each $i=0, \dotsc, k$, there is a family of symplectic structures $\eta_i(s)$ ($t_i \le s \le t_{i+1}$) such that
\begin{enumerate}[(i)]
\item $\eta_i(s)$ varies smoothly on $s \in [t_i,t_{i+1}]$,
\item $[\eta_i(s)] = \theta(s)$ for all $s \in [t_i,t_{i+1}]$,
\item in particular, $\eta_0(0)=\omega_{p,q}$.
\item $K_{\eta_k(1)} \cdot \eta_k(1) < 0$
\end{enumerate}
\end{lemma}

\begin{proof}
Fix $t \in [0,1]$. One may choose an exact form $d\alpha(t)$ such that $\lambda(t)+d\alpha(t)$ is a symplectic form which is cohomologous to $\lambda(t)$; that is, $[\lambda(t)+d\alpha(t))]=[\lambda(t)]=\theta(t)$. In particular, we take $d\alpha(0)$ by $d\alpha(0)=0$ because $\lambda(0)=\omega_{p,q}$ is already a symplectic form. We then translate the line $L_{\lambda}$ parallelly by $d\alpha(t)$ so that we get another line segment in $\Omega^2(X_{p,q})$:
\begin{equation*}
L_{\lambda}+d\alpha(t) = \{\lambda(s)+d\alpha(t) \mid 0 \le s \le 1\}.
\end{equation*}

For $s \in [0,1]$, let $\eta(s) = \lambda(s) + d\alpha(t)$. Notice that the space of symplectic forms on $X_{p,q}$ is open in $\Omega^2(X_{p,q})$; cf. Li~\cite[Lemma~2.2]{Li-2008}. So there is an open neighborhood of the symplectic form $\eta(t)$ in $\Omega^2(X_{p,q})$ which consists only of symplectic forms on $X_{p,q}$. Therefore, since $\eta(t)$ is a symplectic form, a small segment containing $\eta(t)$ (say $\Omega(t)$) of the line $L_{\lambda}+d\alpha(t)$ is contained in the open neighborhood. Explicitly, there is a positive number $\epsilon_t$ and an open interval $I(t) \subset [0,1]$ such that $\eta(s)$ is a symplectic form for any $s \in I(t)$, where $I(t)$ is defined by $I(t)=(t-\epsilon_t, t+\epsilon_t)$ for $t \neq 0$; or $I(0)=[0,\epsilon_0)$; or $I(1)= (1-\epsilon_1, 1]$. We then define explicitly $\Omega(t)$ again by
\begin{equation*}
\Omega(t) = \{\eta(s) \in L_{\lambda}+d\alpha(t) \mid s \in I(t)\}.
\end{equation*}
Then $\Omega(t)$ is a family of symplectic forms $\eta(s)$ varying smoothly on $s \in I(t)$ such that $[\eta(s)]=\theta(s)$.

However the unit interval $[0,1]$ is compact. So there are finitely many intervals $I(t)$'s that cover $[0,1]$. Hence one may choose an increasing finite sequence $0=t_0 < t_1 < \dotsb < t_k < t_{k+1}=1$ such that each interval $[t_i, t_{i+1}]$ ($i=0, \dotsc, k$) is contained in an interval $I(t)$ for some $t$. In particular, we may choose $t_1$ and $t_k$ so that  $[t_0, t_1] \subset I(0)$ and $[t_k, t_{k+1}] \subset I(1)$, respectively.  For each interval $[t_i, t_{i+1}]$, we choose only one such interval $I(t)$. Obviously, we choose $I(0)$ for $[t_0, t_1]$ and $I(1)$ for $[t_k, t_{k+1}]$. We then define a subfamily $\Omega_i$ of $\Omega(t)$ by
\begin{equation*}
\Omega_i = \{\eta(s) \in \Omega(t) \mid s \in [t_i, t_{i+1}]\}.
\end{equation*}
We denote a symplectic form $\eta(s) \in \Omega_i$ by $\eta_i(s)$. Then $\eta_i$ is the desired family; that is, each $\eta_i(s)$ is a symplectic form that varies smoothly on $s \in [t_i, t_{i+1}]$ such that $[\eta_i(s)]=\theta(s)$ for all $s \in [t_i, t_{i+1}]$.
\end{proof}

\begin{lemma}\label{lemma:B}
Two symplectic forms $\eta_i(t_{i+1})$ and $\eta_{i+1}(t_{i+1})$ in Lemma~\ref{lemma:A} are symplectictomorphic. That is, there exists a orientation-preserving diffeomorphism $\phi \colon X_{p,q} \to X_{p,q}$ such that $\phi^{\ast}(\eta_{i+1}(t_{i+1}))=\eta_i(t_{i+1})$.
\end{lemma}

\begin{proof}
This lemma is a simple consequence of the well-known fact that two cohomologous symplectic forms on a \emph{rational} 4-manifolds are symplectomorphic (cf. Li~\cite[Theorem~3.18]{Li-2008}).
\end{proof}

The symplectic structures in Proposition~\ref{proposition:Key-proposition} may be constructed by gluing piece by piece the symplectic structures in Lemma~\ref{lemma:A}. Recall that two symplectic forms are \emph{deformations equivalent} if they are connected by a smooth path of symplectic forms and the deformation equivalence relation is an equivalence relation;  cf. Li~\cite[p.4]{Li-2008}.

\begin{proof}[Proof of Proposition~\ref{proposition:Key-proposition}]
We would like to construct a smooth family $\omega_t$ of symplectic structures on $X_{p,q}$ starting with $(X_{p,q}, \omega_{p,q})$ and ending with $(X_{p,q}, \omega_1)$ with $K_{\omega_1} \cdot \omega_1 < 0$.

By Lemma~\ref{lemma:A}, for each $t_j \le t \le t_{j+1}$, there is a smooth family of symplectic structures $\beta_j(t)$ such that $\beta_j(t_j)=\eta_j(t_j)$ and $\beta_j(t_{j+1})=\eta_j(t_{j+1})$. By Lemma~\ref{lemma:B}, there exists a orientation-preserving diffeomorphism $\phi_j \colon X_{p,q} \to X_{p,q}$ such that $\phi_j^{\ast}(\eta_j(t_j))=\eta_{j-1}(t_j)$. That is, $(\phi_j^{-1})^{\ast}(\eta_{j-1}(t_j)) = \eta_j(t_j)$. In particular, $(\phi_1^{-1})^{\ast}(\beta_0(0))$ is deformation-equivalent to $(\phi_1^{-1})^{\ast}(\beta_0(t_1))(=\eta_1(t_1))$ along $(\phi_1^{-1})^{\ast}(\beta_0(t))$. Therefore $(\phi_1^{-1})^{\ast}(\eta_0(0))$ is deformation-equivalent to $(\phi_1^{-1})^{\ast}(\eta_0(t_1))(=\eta_1(t_1))$ along $(\phi_1^{-1})^{\ast}(\beta_0(t))$ ( $t_0 \le t \le t_1$) and $\eta_1(t_1)$ is deformation-equivalent to $\eta_1(t_2)$ along $\beta_1(t)$ ($t_1 \le t \le t_2$). Since the deformation-equivalent relation is an equivalence relation, we have
\begin{equation*}
(\phi_1^{-1})^{\ast}(\eta_0(0)) \sim \eta_1(t_2)
\end{equation*}
where $\sim$ means that they are deformation-equivalent.

Repeat this argument, we may conclude that there is an orientation-preserving diffeomorphism $\Phi^{-1} \colon X_{p,q} \to X_{p,q}$ (that is given by the composition of $\phi_j^{-1}$'s) such that $(\Phi^{-1})^{\ast}(\eta_0(0))$ is deformation-equivalent to $\eta_k(1)$. Set $\omega_1 = \Phi^{\ast}(\eta_k(1))$. Then $\eta_0(0)(=\omega_{p,q})$ is deformation-equivalent to $\omega_k$. Since $K_{\eta_k(1)} \cdot \eta_k(1) < 0$, we have $K_{\omega_1} \cdot \omega_1 < 0$ as desired.
\end{proof}

\section{Rational symplectic 4-manifolds after rational blow-down surgeries}\label{section:main-theorem}

We prove that if a symplectic 4-manifold $X$ becomes a rational 4-manifold after applying a symplectic rational blow-down surgery, then $X$ is originally rational; Theorem~\ref{theorem:Main}. In order to show rationality, we apply the following well-known fact:

\begin{proposition}[{McDuff-Salamon~\cite[Corollary~1.4]{McDuff-Salamon-1996}}]\label{proposition:McDuff-Salamon}
Let $X$ be a minimal symplectic 4-manifold with $b_2^+=1$. The following are equivalent.

    \begin{enumerate}[(a)]
    \item The 4-manifold $X$ admits a symplectic structure $\omega$ with $K \cdot \omega < 0$.

    \item The 4-manifold $X$ is either rational or ruled.
    \end{enumerate}
\end{proposition}

The above Proposition~\ref{proposition:McDuff-Salamon} extends to the non-minimal case. Then the last condition (ii) in the above proposition reads as follows: The 4-manifold $X$ is a blow up of a rational (or ruled) 4-manifold.

\begin{proposition}[{Gay-Stipsicz~\cite[Lemma~3.3]{Gay-Stipsicz-2009}}]\label{proposition:Gay-Stipsicz}
Suppose that the bilinear form $\bilinear{x,y}$ is given by the negative definite symmetric matrix $Q$ with only nonnegative off-diagonals in the basis $\{E_i\}$. If for a vector $x$ the inequalities $\bilinear{x,E_i} \le 0$ ($i=1,\dotsc,n$) are all satisfied, then all coordinates of $x$ are nonnegative.
\end{proposition}

As a simple application:

\begin{lemma}\label{lemma:P^-1<=0}
Let $P$ be a negative definite $n \times n$ matrix. Let $x$ be a $n \times 1$ column matrix with $x \ge 0$, that is, all of entries of $x$ are non-negative. Then $P^{-1} x \le 0$, that is, all entries of $P^{-1} x$ are non-positive.
\end{lemma}

\begin{proof}
We define a bilinear form $\bilinear{x,y}$ by $\bilinear{x,y} = x^t P y$. Let $E_i^*$ be the $i$-th column of the matrix $-P^{-1}$. Then $\bilinear{E_i^*, E_j}=-\delta_{ij}$ for all $j=1,\dotsc,n$ because $(-P^{-1})P=-I_n$. By Propostion~\ref{proposition:Gay-Stipsicz}, all cooordinates of $E_i^*$ are nonnegative. Hence all entries of $P^{-1}$ are non-positive. So if $x \ge 0$, then $P^{-1} x \le 0$, as asserted.
\end{proof}

\begin{proposition}[{J. Park~\cite[p.365]{JPark-1997}}]
Let $P$ be the plumbing matrix for $C_{p,q}$ with respect to the basis $\{u_i\}$. Then the intersection form on $H^2(C_{p,q};\mathbb{Q})$ with respect to the dual basis $\{r_i\}$ (i.e., $\langle r_i, u_j \rangle = \delta_{ij}$) is given by $Q := (r_i \cdot r_j) = P^{-1}$.
\end{proposition}

\begin{theorem}\label{theorem:Main}
Let $(X,\omega_X)$ be a symplectic 4-manifold. Suppose that a configuration $C_{p,q}$ is symplectically embedded in $X$; that is, all 2-spheres $u_i$ in $C_{p,q}$ are symplectically embedded in $X$ and they intersect positively. Let $(X_{p,q}, \omega_{p,q})$ be the symplectic rational blow-down of $(X,\omega_X)$ along $C_{p,q}$. If $X_{p,q}$ is rational, then so is $X$.
\end{theorem}

\begin{proof}
By perturbing the almost complex structure $J$ on $X$ if necessary, we may assume that all $u_i$'s in $C_{p,q}$ are pseudoholomorphic curves. Since $B_{p,q}$ is a rational homology ball, we have
\begin{equation*}
K_X \cdot \omega_X = K_{X_{p,q}} \cdot \omega_{p,q} + K_X|_{C_{p,q}} \cdot \omega_X|_{C_{p,q}}.
\end{equation*}

We first show that
\begin{equation}\label{equation:K_X|C_pq.omega_X|C_pq<=0}
K_X|_{C_{p,q}} \cdot \omega_X|_{C_{p,q}} \le 0
\end{equation}

Notice that $K_X|_{C_{p,q}} = \sum (K_X \cdot u_i) r_i$ and $\omega_X|_{C_{p,q}} = \sum (\omega_X \cdot u_i) r_i$. Since each $u_i$ is a pseudoholomorphic curve, the adjunction inequality holds for $u_i$; hence we have
\begin{equation}\label{equation:K_X.u_i>=0}
K_X \cdot u_i = -2 - u_i^2 \ge 0
\end{equation}
On the other hand, $\omega_X \cdot u_i (= \int_{u_i} \omega_X)$ is the area of $u_i$ in $X$; hence we also have
\begin{equation}\label{equation:omega_X.u_i>0}
\omega_X \cdot u_i > 0
\end{equation}
Note that
\begin{equation*}
K_X|_{C_{p,q}} \cdot \omega_X|_{C_{p,q}} = \left( \sum (K_X \cdot u_i)r_i\right) P^{-1} \left( \sum (\omega_X \cdot u_j) r_j \right)
\end{equation*}
By Lemma~\ref{lemma:P^-1<=0}, $P^{-1}r_j \le 0$. Then $r_iP^{-1}r_j \le 0$. Therefore it follows by Equations~\eqref{equation:K_X.u_i>=0} and \eqref{equation:omega_X.u_i>0} that $K_X|_{C_{p,q}} \cdot \omega_X|_{C_{p,q}} \le 0$, as asserted.

\textit{Case 1}: $K_{X_{p,q}} \cdot \omega_{p,q} < 0$.

By Equation~\eqref{equation:K_X|C_pq.omega_X|C_pq<=0}, we have $K_X \cdot \omega_X < 0$. Hence $X$ is rational by Proposition~\ref{proposition:McDuff-Salamon}.

\textit{Case 2}: $K_{X_{p,q}} \cdot \omega_{p,q}  \ge 0$ although $X_{p,q}$ is rational.

By Proposition~\ref{proposition:Key-proposition}, there is a smooth family of symplectic structures $\omega_t$ ($t \in [0,1]$) on $X_{p,q}$ such that $\omega_0=\omega_{p,q}$ and $K_1 \cdot \omega_1 < 0$.

Let $X_0 = X_{p,q} - B_{p,q}$. Then $X_{p,q} = X_0 \cup_{\partial B_{p,q}} B_{p,q}$, where $\partial B_{p,q}$ is the Lens space $L(p^2,pq-1)$. One may  stretch the attachment part $L (=L(p^2,pq-1))$ out; that is,
\begin{equation*}
X_{p,q} = X_0 \bigcup_{L \times \{0\}} (L \times [0,1]) \bigcup_{L \times \{1\}} B_{p,q}.
\end{equation*}
We define a symplectic structure $\overline{\omega}$ on $X_{p,q}$ as follows: For a point $x \in X_{p,q}$,
\begin{equation*}
\overline{\omega}_x = \begin{cases}
(\omega_0)_x & \text{if $x \in B_{p,q}$}, \\
(\omega_t)_x & \text{if $x \in L \times (0,1)$}, \\
(\omega_1)_x & \text{if $x \in X_0$}
\end{cases}
\end{equation*}
If necessary, one can deform the above $\overline{\omega}$ slightly so that $\overline{\omega}$ varies smoothly. With this new symplectic structure $\overline{\omega}$ on $X_{p,q}$, the boundary $\partial B_{p,q}$ has a Milnor fillable contact structure. Furthermore note that
\begin{equation*}
\overline{K} \cdot \overline{\omega} < 0,
\end{equation*}
where $\overline{K}$ is the canonical class on $X_{p,q}$ associated to the symplectic structure $\overline{\omega}$. Indeed, $X_{p,q} - X_0 (= (L \times [0,1]) \bigcup_{L \times \{1\}} B_{p,q})$ is diffeomorphic to a rational homology ball $B_{p,q}$. So we have
\begin{equation*}
\overline{K} \cdot \overline{\omega} = K_1|_{X_0} \cdot \omega_1|_{X_0} = K_1 \cdot \omega_1 < 0.
\end{equation*}

We can apply a symplectic rational blow-up to $X_{p,q}$ because $\partial B_{p,q}$ has a Milnor fillable contact structure so that we get a symplectic 4-manifold $X'$. Since $\overline{K} \cdot \overline{\omega} < 0$, it is the same situation as Case (1); hence, $X'$ is rational. Therefore $X$ is also rational because $X$ is diffeomorphic to $X'$.
\end{proof}

It is not difficult to generalize the above theorem to the case that $X$ has multiple $C_{p,q}$'s.

\begin{corollary}\label{corollary:multiple-cpq}
Let $(X,\omega_X)$ be a symplectic 4-manifold. Suppose that $X$ contains disjoint symplectic $C_{p,q}$'s for various $p,q$'s. If the symplectic rational blow-down $Y$ of $X$ along the $C_{p,q}$'s is a rational symplectic 4-manifold, then $X$ is originally rational.
\end{corollary}

\begin{remark}
It may be said that a symplectic rational blow-up of a rational symplectic $4$-manifold is again a rational symplectic $4$-manifold.
\end{remark}

\begin{remark}
A cyclic quotient surface singularity of type $\frac{1}{p^2}(1,pq-1)$ (which is called a \emph{singularity of class $T$}) has a Milnor fiber diffeomorphic to $B_{p,q}$. So a rational blow-down surgery may be regarded as a surgery that replaces a regular neighborhood of the collection of the exceptional 2-spheres of the cyclic quotient surface singularity of type $\frac{1}{p^2}(1,pq-1)$ with its Milnor fiber diffeomorphic to $B_{p,q}$.
\end{remark}

\section{Degenerations of rational complex surfaces}\label{section:degenerations}

Let $X$ be a normal projective surface over $\mathbb{C}$. Suppose that there is a flat deformation $f \colon \mathcal{X} \to \Delta$ of $X$ over a small disk $\Delta = \{t \in \mathbb{C} \mid \lvert t \rvert < \epsilon\}$ such that $X$ is the special fiber $X_0(=f^{-1}(0))$ and a general fiber $X_t=f^{-1}(t)$($t \neq 0$) is a smooth surface. In such a situation $X$ is called a \emph{degeneration} of the family $\{X_t \mid t \in \Delta-\{0\}\}$.

Bădescu~\cite{Badescu-1986} characterizes a degeneration of a family of rational surfaces:

\begin{proposition}[{Bădescu~\cite[Theorem~1]{Badescu-1986}}]\label{proposition:Badescu}
Let $X$ be a degeneration of a family $\{X_t \mid t \in \Delta-\{0\}\}$. Assume that $X_t$ is a smooth rational surface with $b_2(X_t) \le 10$ for some (and hence, for all) $t \in \Delta-\{0\}$). Let $\widetilde{X}$ is the minimal resolution of $X$. Then either (a) $\widetilde{X}$ is a rational surface and $X$ contains at most rational singularities; or (b) $\widetilde{X}$ is a ruled surface with irregularity $q > 0$ and $X$ contains precisely one non-rational singular point.
\end{proposition}

Applying Theorem~\ref{theorem:Main} (or Corollary~\ref{corollary:multiple-cpq}), one may generalize Bădescu's result without the hypothesis on the second Betti number of general fibers $X_t$, but, under some additional mild condition on the central fiber $X_0$.

\begin{corollary}\label{corollary:degeneration}
Let $X$ be a degeneration of a family $\{X_t \mid t \in \Delta-\{0\}\}$ of rational surfaces. Suppose that the degeneration $X$ has at most quotient surface singularities. Then its minimal resolution $\widetilde{X}$ is also a rational surface.
\end{corollary}

\begin{proof}
Topologically, a general fiber $X_t$ is diffeomorphic to the symplectic 4-manifold constructed from the minimal resolution $\widetilde{X}$ by replacing small neighborhoods of the collection of the exceptional $2$-spheres of each singular points of $X$ with their Milnor fibers corresponding to the given smoothing $f \colon \mathcal{X} \to \Delta$.

Notice that any minimal symplectic fillings of a quotient surface singularity can be obtained from a small neighborhood of the collection of its exceptional 2-spheres by applying a sequence of rational blow-down surgeries (See Bhupal-Ozbagci~\cite{Bhupal-Ozbagci-2016} and H. Choi-J. Park~\cite{Choi-Park-2020}, or refer H. Choi-H. Park-D. Shin~\cite{Choi-Park-Shin-2021}). Since Milnor fibers of a quotient surface singularity is naturally a minimal symplectic filling of the singularity, a general fiber $X_t$ is diffeomorphic to the symplectic 4-manifold obtained from the minimal resolution $\widetilde{X}$ by a sequence of rational blow-down surgeries. Therefore if $X_t$ is rational then so is  $\widetilde{X}$ by Corollary~\ref{corollary:multiple-cpq}.
\end{proof}

\end{document}